\documentclass[sn-basic]{sn-jnl}
\usepackage{graphicx}%
\usepackage{amsmath,amssymb,amsfonts}%
\usepackage{amsthm}%
\usepackage{bbm}
\usepackage{mathrsfs}%
\usepackage{xcolor}%
\usepackage{textcomp}%
\usepackage{manyfoot}%
\usepackage{etoolbox}

\theoremstyle{thmstyleone}%
\newtheorem{theorem}{Theorem}[section]%
\newtheorem{lemma}{Lemma}[section]%
\setcitestyle{citesep={,}}
\makeatletter
\patchcmd{\@maketitle}{Contributing authors:\ }{E-mails:\ }{}{}
\def\email#1{\global\advance\emailcnt by 1\relax%
  \ifnum\emailcnt>1\g@addto@macro\authemail{;\ }\fi%
  \g@addto@macro\authemail{\textcolor{blue}{#1}}}
\makeatother

\begin{document}
\title[Asymptotic Minimax Estimation under Global--Local Priors]{Asymptotic Minimax Estimation under Global--Local Priors}

\author{\fnm{Rhitankar} \sur{Bandyopadhyay}}\email{r.bandyopadhyay@ufl.edu}
\author{\fnm{Malay} \sur{Ghosh}}\email{ghoshm@ufl.edu}
\affil{\orgdiv{Department of Statistics}, \orgname{University of Florida}, \orgaddress{\city{Gainesville}, \state{FL} \postcode{32611}, \country{USA}}}

\abstract{Global-local priors, often also referred to as shrinkage priors, have proved to be a very effective tool for the analysis of high dimensional data under sparsity. Asymptotic theoretical properties of such priors, studied under various scenarios, are now available in the literature. However, to our knowledge, theoretical guarantees of such priors provided so far, involve the assumption of known sample variance. The present paper relaxes this assumption,
and carries out the analysis with a prior assigned to the error variance as well. In the process, some new tail bounds for shrinkage factors are developed, and these results are then utilized in providing asymptotic minimax rates for the posterior means of the parameters of interest.}


\maketitle

\nocite{castillo2012}

\section{Introduction}

\noindent
Global-local shrinkage priors, introduced by \citet{carvalho2009, carvalho2010} and followed subsequently in a multitude of other papers, have now
become quite popular for simultaneous estimation and multiple testing of normal means under sparsity. Applications abound, for example, in microarray  experiments, image reconstruction, small area estimation and also in a vast
majority of other problems in science, engineering, finance, sociology, just to name a few. The success of such priors does not rest just on identifying many zeroes but finding as well a few signals in the midst of multiple noises. This is one of the most fundamental issues in many scientific
disciplines, high energy physics being one prime example.\\

\noindent
The existing literature has not just been concentrated in the estimation and testing of normal means under sparsity but also for variable selection in high dimensional regression models where the number of regressors far outnumber the sample size, but only a few have a real impact on the 
predictand in a given situation. An example is Type I diabetes, where out of millions of genes, only a handful has been recognized, caausing the disease.\\

\noindent
Global-local priors assigned to normal means for either estimation or testing under sparsity are hierarchical. The first stage involves normal distributions with zero means, while the variances consist of two components. One global
component, intended to shrink all the means towards zeroes are very effective for handling sparsity. The others, the local components, are distinct for each individual mean and are assigned heavy tailed priors in the next stage, thereby preventing genuine signals being shrunk towards zeroes.\\

\noindent
There are several excellent articles with the above framework for both estimation and testing of normal means. Among others, mention may be made of \citet{vanderpas2014}, \citet{vanderpas2016} and \citet{ghosh2017}. A multivariate
extension of the work of \cite{ghosh2017} appears in \citet{qin2024}.\\

\noindent
The present article also begins with the same framework of the above authors but with an added feature. To our knowledge, the previous work typically assumes known sample variances for any analytical study. In this aricle, we assign instead prior distributions to the same and carry out the
corresponding analysis. We concentrate on estimation of normal means, develop some new concentration inequalities and prove asymptotic minimax rates of posterior means as estimators of population means. \\

\noindent
We consider two types of priors for the error variance. In Section~\ref{sec:results-exponential} of this paper, we assign an exponential prior and provide the two main results dealing with asymptotic minimaxity. In Section~\ref{sec:results-inversegamma}, we use the Inverse Gamma prior and provide the corresponding minimaxity results. Lemmas and proofs related to the results of Section~\ref{sec:results-exponential} and Section~\ref{sec:results-inversegamma} are given in Section~\ref{sec:proofs-exponential} and Section~\ref{sec:proofs-inversegamma} respectively. Some final remarks are made in Section~\ref{sec:remarks}.\\

\section{The Results: Exponential prior}\label{sec:results-exponential}

\noindent
Following the lines of \citet{ghosh2017}, we begin with the model\\
(i) $X_i|\theta_i,\sigma_i^2\stackrel{\mbox{ind}}
{\sim}\mbox{N}(\theta_i,\sigma_i^2)$;
$i=1.\ldots,n$,\\
(ii) $\theta_i|\sigma_i^2,\lambda_i^2,\tau^2\stackrel{\mbox{ind}}{\sim}
\mbox{N}(0,\sigma_i^2\lambda_i^2\tau^2)$; $i=1.\ldots,n$,\\
(iii) $\lambda_1^2,\ldots,\lambda_n^2,\sigma_1^2,\ldots,\sigma_n^2$ are mutually
independent with $\Pi(\lambda_i^2)=K(\lambda_i^2)^{-a-1}L(\lambda_i^2)$,
$a>0$ and $L$ is a slowly varying function,\\
(iv) an exponential prior $\Pi(\sigma_i^2)=(1/2)\exp(-\sigma_i^2/2)$ on the error variance.\\

\noindent
As in \citet{ghosh2017}, we take $\tau\equiv\tau_n$ as the tuning parameter. The new feature in this model is introduction of a prior on the $\sigma_i^2$
instead of treating them as constants. The TPBN prior (see e.g. \citet{ghosh2016}) is given by $\Pi(\lambda_i^2)\propto(\lambda_i^2)^{b-1}  
(1+\lambda_i^2)^{-a-b}=(\lambda_i^2)^{-a-1}(1+\lambda_i^{-2})^{-a-b}$ is of 
the form (iii) with $L(\lambda_i^2)=(1+\lambda_i^{-2})^{-a-b}$, a bounded slowly varying function. The special case $a=b=1/2$ yields the now celebrated 
horseshoe prior. \\

\noindent
We need the following assumption on the slowly varying $L$.\\

\noindent
$L$ is nondecreasing in its argument with
\begin{equation} 
0<L(u)\leq M<\infty.
\label{eq:1}
\end{equation}
Assuming squared error loss, the Bayes estimator of $\theta_i$ is given by
\begin{equation}
\hat{\theta}_i=\mathbb{E}(1-\kappa_i|X_i)X_i,
\label{eq:2}
\end{equation}
where $\kappa_i=(1+\lambda_i^2\tau^2)^{-1}$.
We are interested in finding the asymptotic minimax  rate of the estimator 
given in (\ref{eq:2}).
To this end, we assume that the true means $\theta_{0i}$ are sparse in the 
nearly black sense, namely,
\begin{equation}
(\theta_{01},\ldots,\theta_{0n})\in L_0(q_n)=
\{\theta_{01},\ldots,\theta_{0n}:\sum_{i=1}^n I_{[\theta_{0i}\neq 0]}\leq q_n\}.
\label{eq:3}
\end{equation}
The first main result of this section provides asymptotic minimax rate of
the Bayes estimators given in (\ref{eq:2}).\\

\noindent
\begin{theorem}\label{thm:exp1}
Under the assumption given in (\ref{eq:3}), with further assumptions (i)  
$\mbox{max}_{1\leq i\leq n}\sigma_{0i}^2=O(1)$ and
(ii) $q_n=o(n)$, choosing $\tau_n=(q_n/n)^{1+\epsilon_n}\log(n/q_n)$, where
$\epsilon_n\rightarrow 0$ as $n\rightarrow\infty$,
\begin{equation}
\mbox{lim}_{n\rightarrow\infty}
\mbox{sup}_{(\theta_{01},\ldots,\theta_{0n})\in L_0(q_n)}
\frac{\sum_{i=1}^n \mathbb{E}_{i}(\hat{\theta}_i-\theta_{0i})^2}
{q_n\log(n/q_n)}\leq C,
\label{eq:theorem1}
\end{equation}
where $C(>0)$ is a constant depending on $\sigma_{0i}^2$.
\end{theorem}

\noindent
The second theorem of this section provides contraction rates of the 
posterior means $\hat{\theta}_i$. It says that the posterior distributions
of the parameters center around the corresponding Bayes estimators at least
as fast as the minimax rate.\\

\noindent
\begin{theorem}\label{thm:exp2}
Under the assumptions of Theorem~\ref{thm:exp1} and the additional assumption that
$\mbox{max}_{1\leq i\leq n}|\theta_{0i}|=O(1)$,
\begin{equation}
\mbox{lim}_{n\rightarrow\infty}
\mbox{sup}_{(\theta_{01},\ldots,\theta_{0n})\in L_0(q_n)}
\sum_{i=1}^n \mathbb{P}_{0i}[(\hat{\theta}_i-\theta_{0i})^2>M_n{q_n\log(n/q_n)]}=0.
\label{eq:theorem2}
\end{equation}
\end{theorem}

\section{The Results: Inverse Gamma prior}\label{sec:results-inversegamma}

\noindent
Following the lines of \citet{ghosh2017} again, we begin with the model\\
(i) $X_i|\theta_i,\sigma_i^2\stackrel{\mbox{ind}}
{\sim}\mbox{N}(\theta_i,\sigma_i^2)$;
$i=1.\ldots,n,$\\
(ii) $\theta_i|\sigma_i^2,\lambda_i^2,\tau^2\stackrel{\mbox{ind}}{\sim}
\mbox{N}(0,\sigma_i^2\lambda_i^2\tau^2)$; $i=1.\ldots,n$,\\
(iii) $\lambda_1^2,\ldots,\lambda_n^2,\sigma_1^2,\ldots,\sigma_n^2$ are mutually
independent with $\Pi(\lambda_i^2)=K(\lambda_i^2)^{-a-1}L(\lambda_i^2)$,
$a>0$ and $L$ is a slowly varying function,\\
(iv) an inverse gamma prior $\Pi(\sigma_i^2)=(\beta^\alpha/\Gamma(\alpha)) (\sigma_i)^{-\alpha-1} \exp(-\beta/\sigma_i^2)$ on the error variance.\\ 

\noindent
Equations~(\ref{eq:1})-(\ref{eq:3}) hold in similar fashion as described in Section~\ref{sec:results-exponential}. The first main result of this section provides asymptotic minimax rate of
the Bayes estimators given in (\ref{eq:2}).\\

\noindent
\begin{theorem}\label{thm:ig1}
Under the assumption given in (\ref{eq:3}), with further assumptions (i)  
$\mbox{max}_{1\leq i\leq n}\sigma_{0i}^2=O(1)$ and
(ii) $q_n=o(n)$, choosing $\tau_n=(q_n/n)^{\frac{1}{2\eta}(1+\epsilon_n)}$, where
$\epsilon_n\rightarrow 0$ as $n\rightarrow\infty$,
\begin{equation}
\mbox{lim}_{n\rightarrow\infty}
\mbox{sup}_{(\theta_{01},\ldots,\theta_{0n})\in L_0(q_n)}
\frac{\sum_{i=1}^n \mathbb{E}_{i}(\hat{\theta}_i-\theta_{0i})^2}
{q_n\log(n/q_n)}\leq C,
\label{eq:theorem1}
\end{equation}
where $C(>0)$ is a constant depending on $\sigma_{0i}^2$ and $a+1<\alpha<a+2$.
\end{theorem}

\noindent
The second theorem of this section provides contraction rates of the 
posterior means $\hat{\theta}_i$. It says that the posterior distributions
of the parameters center around the corresponding Bayes estimators at least
as fast as the minimax rate.\\

\noindent
\begin{theorem}\label{thm:ig2}
Under the assumptions of Theorem~\ref{thm:ig1} and the additional assumption that
$\mbox{max}_{1\leq i\leq n}|\theta_{0i}|=O(1)$,
\begin{equation}
\mbox{lim}_{n\rightarrow\infty}
\mbox{sup}_{(\theta_{01},\ldots,\theta_{0n})\in L_0(q_n)}
\sum_{i=1}^n \mathbb{P}_{0i}[(\hat{\theta}_i-\theta_{0i})^2>M_n{q_n\log(n/q_n)]}=0.
\label{eq:theorem2}
\end{equation}
\end{theorem}

\section{Proofs: Exponential prior}\label{sec:proofs-exponential}

\noindent
The proofs of both Theorems~\ref{thm:exp1} and \ref{thm:exp2} rest on the following three lemmas.

\begin{lemma}\label{lem:exp1}
Under the prior given in $\textnormal{(i)-(iii)}$, for any $0<\tau<1$,
\[
\mathbb{E}(1-\kappa_i|X_i)\leq C\tau^{2\eta}\exp(|X_i|),
\]
where in the above and hereafter $C(>0)$ is a generic constant, and 
$0<\eta<\mbox{min}(a,1/2)$.\\
\end{lemma}

\begin{proof}
We begin with $X_i|\lambda_i^2,\sigma_i^2\sim\mbox{N}(0,\sigma_i^2
(1+\lambda_i^2\tau^2))$. This leads to
\[
\Pi(\lambda_i^2,\sigma_i^2|X_i)\propto(\sigma_i^2)^{-1/2}(1+\lambda_i^2\tau^2)
^{-1/2}\exp[-X_i^2/\{2\sigma_i^2(1+\lambda_i^2\tau^2)\}](\lambda_i^2)^{-a-1}
L(\lambda_i^2)\exp(-\sigma_i^2/2).
\]
Now integrating over $\sigma_i^2$, and using the transformation 
$z=1/\sigma_i^2$, one gets
\[
\Pi(\lambda_i^2|X_i)\propto(\lambda_i^2)^{-a-1}(1+\lambda_i^2\tau^2)^{-1/2}
L(\lambda_i^2)\int_{0}^{\infty}z^{-3/2}\exp[-(1/2z)\{z^2X_i^2(1+\lambda_i^2\tau^2)+1\}]dz.
\]
Next, recalling that an inverse Gaussian distribution with mean $\mu$ has pdf
\[
f(z)=(2\pi z^3)^{-1/2}\exp[-(z/\mu-1)^2/(2z)]I_{[z>0]},
\]
the above simplifies to
\begin{equation}
\Pi(\lambda_i^2|X_i)\propto(\lambda_i^2)^{-a-1}(1+\lambda_i^2\tau^2)^{-1/2}
L(\lambda_i^2)\exp[-|X_i|/(1+\lambda_i^2\tau^2)^{1/2}].
\label{eq:l1-exp-1}
\end{equation}
Hence, from equation~(\ref{eq:l1-exp-1}),
\begin{eqnarray}
\mathbb{E}(1-\kappa_i|X_i) & = & \mathbb{E}[\lambda_i^2\tau^2(1+\lambda_i^2\tau^2)^{-1}|X_i]
\nonumber\\ & = & \frac{\tau^2\int_{0}^{\infty}\lambda_i^2
(1+\lambda_i^2\tau^2)^{-3/2}\exp[-|X_i|/(1+\lambda_i^2\tau^2)^{1/2}
(\lambda_i^2)^{-a-1}L(\lambda_i^2)d\lambda_i^2}
{\int_{0}^{\infty}(1+\lambda_i^2\tau^2)^{-1/2}
\exp[-|X_i|/(1+\lambda_i^2\tau^2)^{1/2}](\lambda_i^2)^{-a-1}
L(\lambda_i^2)d\lambda_i^2}\nonumber\\ & \leq & \tau^2\exp(|X_i|)N/D,
\label{eq:l1-exp-2}
\end{eqnarray}
where $N=\int_{0}^{\infty}\lambda_i^2(1+\lambda_i^2\tau^2)^{-3/2}
(\lambda_i^2)^{-a-1}L(\lambda_i^2)d\lambda_i^2$ and 
$D=\int_{0}^{\infty}(1+\lambda_i^2\tau^2)^{-1/2}(\lambda_i^2)^{-a-1}
L(\lambda_i^2)d\lambda_i^2$.\\

\noindent
For $0<\tau^2<1$, 
\begin{equation}
D\geq \int_0^\infty (1+\lambda^2)^{-1/2}(\lambda^2)^{-a-1}L(\lambda^2)d\lambda^2.
\label{eq:l1-exp-3}
\end{equation}
The last integral is $\mathbb{E}[(1+\lambda^2)^{-1/2}]$ with respect to the proper
prior $\Pi(\lambda^2)$ introduced earlier. 
Next with $0<\eta<\mbox{min}(a,1/2)$,
\begin{eqnarray}
\lefteqn{ N=(\int_{0}^{1} +\int_{1}^{\infty})\lambda^2
(1+\lambda^2\tau^2)^{-3/2}(\lambda^2)^{-a-1}L(\lambda^2)d\lambda^2}\nonumber\\
 & \leq & [K^{-1}+M\int_{1}^{\infty}(\lambda^2\tau^2)^{-1+\eta}
(1+\lambda^2\tau^2)^{-1/2-\eta}(\lambda^2)^{-a}d\lambda^2\nonumber\\
 & \leq & 
K^{-1}+M\tau^{-2+2\eta}/(a-\eta).
\label{eq:l1-exp-4}
\end{eqnarray}
The lemma now follows from \eqref{eq:l1-exp-1}-\eqref{eq:l1-exp-4}.
\end{proof}

\begin{lemma}\label{lem:exp2}
For $0<\tau<1$, $\mathbb{E}(\kappa_iI_{[\kappa_i>u]}|X_i)\leq C
\tau^{-2a}\exp[-u^{1/2}(1-\delta^{1/2})|X_i|]$, for any $0<u<1$, $0<\delta<1$.
\end{lemma}

\begin{proof}
For $0<\kappa<1$, $0<u<1$ and $0<\delta<1$,
\begin{eqnarray}
\mathbb{E}(\kappa_iI_{[\kappa_i>u|X_i]} & = & \mathbb{E}[(1+\lambda_i^2\tau^2)^{-1/2}
I_{[\lambda_i^2\leq (1-u)/(u\tau^2)]}|X_i]\nonumber\\ & \leq &
\frac{\int_{0}^{(1-u)/(u\tau^2)}(1+\lambda^2\tau^2)^{-3/2}
\exp[-|X_i|(1+\lambda^2\tau^2)^{-1/2}](\lambda^2)^{-a-1}L(\lambda^2)d\lambda^2}
{\int_{(1-u\delta)/(u\delta\tau^2)}^{\infty} (1+\lambda^2\tau^2)^{-1/2}
\exp[-|X_i|(1+\lambda^2\tau^2)^{-1/2}](\lambda^2)^{-a-1}L(\lambda^2)d\lambda^2}
\nonumber\\ & \leq & \exp[-u^{1/2}(1-\delta^{1/2})|X_i|]
\frac{\int_{0}^{(1-u)/(u\tau^2)}(1+\lambda^2\tau^2)^{-3/2}
(\lambda^2)^{-a-1}L(\lambda^2)d\lambda^2}
{\int_{(1-u\delta)/(u\delta\tau^2)}^{\infty}(1+\lambda^2\tau^2)^{-1/2}
(\lambda^2)^{-a-1}L(\lambda^2)d\lambda^2}.\nonumber\\
\label{eq:l2-exp-1}
\end{eqnarray}
Next $\int_{0}^{(1-u)/(u\tau^2)}(1+\lambda_i^2\tau^2)^{-3/2}
(\lambda^2)^{-a-1}L(\lambda^2)d\lambda^2\leq\int_{0}^{\infty}
(\lambda^2)^{-a-1}L(\lambda^2)d\lambda^2=K^{-1}$. Also, since $L$ is 
monotonically nondecreasing and $0<u<1$ and $0<\delta<1$,
\begin{eqnarray*}
\int_{(1-u\delta)/(u\delta\tau^2)}^{\infty}(1+\lambda^2\tau^2)^{-1/2}
(\lambda^2)^{-a-1}L(\lambda^2)d\lambda^2 & \geq & L(1-u/u)
\int_{(1-u\delta)/(u\delta\tau^2)}^{2(1-u\delta)/(u\delta\tau^2)}
(\lambda^2)^{-a-1}d\lambda^2
\nonumber\\ & = & L(1-u/u)\tau^{2a}a^{-1}(1-u\delta/u\delta)^{-a}(1-2^{-a}).
\end{eqnarray*}
In view of the above, the result now follows from \eqref{eq:l2-exp-1}.
\end{proof}

\begin{lemma}\label{lem:exp3}
For $0<\tau<1$, $\mathbb{E}[\kappa_i I_{[\kappa_i\leq u]}|X_i]\leq C/X_i^2$.
\end{lemma}

\begin{proof}
Since $L$ is nondecreasing in its argument and $0<\tau^2<1$,
\begin{eqnarray}
\mathbb{E}[\kappa_i I_{[\kappa_i\leq u]}|X_i] & \leq & 
\frac{\int_{0}^{u} (1-\kappa)^{-a}\kappa^{a+1/2}\exp(-\kappa^{1/2}|X_i|)
L((1-\kappa)/(\kappa\tau^2))d\kappa}
{\int_{0}^u\kappa^{a-1/2}\exp(-\kappa^{1/2}|X_i|)L((1-\kappa)/(\kappa\tau^2))
d\kappa}
\nonumber\\ & \leq & \frac{M}{L(1-u/u)}
\frac{\int_{0}^{u} (1-\kappa)^{-a}\kappa^{a+1/2}\exp(-\kappa^{1/2}|X|)d\kappa}
{\int_{0}^{u}\kappa^{a-1/2}\exp(-\kappa^{1/2}|X|)d\kappa} \nonumber\\ & \leq &
\frac{M(1-u)^{-a}}{L(1-u/u)}
\frac{\int_{0}^{u}\kappa^{a+1/2}\exp(-\kappa^{1/2}|X|)d\kappa}
{\int_{0}^{u}\kappa^{a-1/2}\exp(-\kappa^{1/2}|X|)d\kappa}.
\label{eq:l3-exp-1}
\end{eqnarray}
Now with the substitution $t=\kappa^{1/2}|X|$, ratio of the two integrals in 
the right hand side of equation~(\ref{eq:l3-exp-1}) simplifies to 
\begin{equation}
(1/X_i^2)\frac{\int_{0}^{u^{1/2}|X_i|}t^{2a+2}\exp(-t)dt}
{\int_{0}^{u^{1/2}|X_i|}t^{2a}\exp(-t)dt}.
\label{eq:l3-exp-2}
\end{equation}
Successive integaration by parts yields
\begin{eqnarray}
\int_{0}^{u^{1/2}|X_i|}t^{2a+2}\exp(-t)dt & \leq &
(2a+2)\int_{0}^{u^{1/2}|X_i|}t^{2a+1}\exp(-t)dt \nonumber\\ & \leq &
(2a+2)(2a+1)\int_{0}^{u^{1/2}|X_i|}t^{2a}\exp(-t)dt.
\label{eq:l3-exp-3}
\end{eqnarray}
The result now follows from \eqref{eq:l3-exp-1}-\eqref{eq:l3-exp-3}.
\end{proof}

\noindent
We now utilize the above lemmas to prove Theorems~\ref{thm:exp1} and \ref{thm:exp2}.\\

\noindent
\textbf{Proof of Theorem~\ref{thm:exp1}}. We consider separately the two cases $\{i:\theta_{0i}=0\}$
and $\{i:\theta_{0i}\neq 0\}$. First when $\theta_{0i}=0$, writing
$\mathbb{E}_{0i}$ to denote $\mathbb{E}_{0,\sigma_{0i}^2}$, one gets
\begin{equation}
\mathbb{E}_{0i}[\mathbb{E}(1-\kappa_i|X_i)X_i]^2\leq \mathbb{E}_{0i}
[X_i^2E(1-\kappa_i|X_i)]
\leq C\tau^{2\eta}\mathbb{E}_{0i}[X_i^2\exp(|X_i|)].
\label{eq:t1-exp-1}
\end{equation}
We then bound
\begin{eqnarray}
\mathbb{E}_{0i}[X_i^2\exp(|X_i|] & \leq &
\mathbb{E}_{0i}[X_i^2\exp(X_i)]+\mathbb{E}_{0i}[X_i^2\exp(-X_i)]
\nonumber\\ & = & 
(\sigma_{0i}^2+\sigma_{0i}^4)[\exp(\sigma_{0i}^2/2)+\exp(-\sigma_{0i}^2/2)].
\label{eq:t1-exp-2}
\end{eqnarray}
Since $0\leq\eta\leq\mbox{min}(1/2,a)$, $\tau_n^{2\eta}\leq\tau_n$. 
Thus with the
given choice of $\tau_n$, it follows from equations~(\ref{eq:t1-exp-1}) and (\ref{eq:t1-exp-2}) that 
\begin{equation}
\begin{aligned}
&\sum_{i:\theta_{0i}=0} \mathbb{E}_{0i}[\mathbb{E}(1-\kappa_i|X_i)X_i]^2\\
&\leq
C(q_n/n)^{1+\epsilon_n}\log(n/q_n)\sum_{i:\theta_{0i}=0}
(\sigma_{0i}^2+\sigma_{0i}^4)[\exp(\sigma_{0i}^2/2)+\exp(-\sigma_{0i}^2/2)]\\
\label{eq:t1-exp-3}
\end{aligned}
\end{equation}
In view of  the fact that $\mbox{max}_{1\leq i\leq n}\sigma_{0i}^2=O(1)$,
one now gets
\begin{equation}
\frac{\sum_{\{i:\theta_{0i}=0\}} \mathbb{E}_{0i}[\mathbb{E}(1-\kappa_i|X_i)X_i]^2)}
{q_n\log(n/q_n)}\leq C(n-q_n)\frac{(q_n/n)^{1+\epsilon_n}\log(n/q_n)}
{q_n\log(n/q_n)}\rightarrow 0 \; \mbox{as} \; n\rightarrow\infty.
\label{eq:t1-exp-4}
\end{equation}

\noindent
Next we consider the case when $\theta_{0i}\neq 0$. First writing $E_i$ for 
$E_{\theta_{0i},\sigma_{0i}^2}$, we use the inequality
\begin{eqnarray}
\mathbb{E}_i(\hat{\theta}_i-\theta_{0i})^2 & = &
\mathbb{E}_i[X_i-\theta_{0i}-\mathbb{E}(\kappa_i|X_i)X_i]^2
\nonumber\\ & \leq & \sigma_{0i}^2+\mathbb{E}_i
[\mathbb{E}(\kappa_i|X_i)X_i^2+2 \sigma_{0i}[\mathbb{E}(\kappa_i|X_i)X_i^2]^{1/2}].
\label{eq:t1-exp-5}
\end{eqnarray}
Next observe that $\mathbb{E}[(\kappa_i I_{[\kappa_i\leq u]}|X_i)X_i^2]\leq C$ from
Lemma~\ref{lem:exp3} so that $\frac{\sum_{\{\theta_{0i}\neq 0\}}\mathbb{E}_i[\mathbb{E}(\kappa_i|X_i)X_i^2]}
{q_n\log(n/q_n)}\leq (Cq_n)/(q_n\log(n/q_n))
\rightarrow 0$ as $n\rightarrow\infty$.\\

\noindent
It remains to consider  $\mathbb{E}_i[\hat{\theta}_i-\theta_{0i})^2 I_{[\kappa_i>u]}|X_i]$.
To this end, first writing $b_n=(1+\rho)4a/[u^{1/2}(1-\delta)^{1/2}]
[\log(n/q_n)]^{1/2}$, $\rho>0$, it follows from Lemma~\ref{lem:exp2} and the 
fact that for any fixed $s>0$,  $y^2\exp(-sy)$  
is decreasing in $y$ for $y>2/s$, for large $n$,
\begin{eqnarray}
\mathbb{E}_i[\mathbb{E}(\kappa_i|X_i)X_i^2 
I_{[\kappa_i>u]}I_{|X_i|>b_n}] &\leq&  C(\tau_n)^{-2a} b_n^{-2a(1+\rho)}
 \nonumber\\ & = &
C(q_n/n)^{\rho-\epsilon_n}[\log(n/q_n)]^2.
\label{eq:t1-exp-6}
\end{eqnarray}
This leads to the inequality 
$\frac{\mathbb{E}_i[\mathbb{E}(\kappa_i|X_i)X_i^2 
I_{[\kappa_i>u]}I_{|X_i|>b_n}]}{q_n\log(n/q_n)}\rightarrow 0$ as 
$n\rightarrow\infty$.
Now noting that $|X_i|\exp[-u^{1/2}(1-\delta^{1/2})|X_i|]\leq
[u^{1/2}(1-\delta^{1/2})]^{-1}$, one gets the inequality
\begin{equation}
\mathbb{E}_i[\mathbb{E}(\kappa_i|X_i)X_i^2 I_{[\kappa_i>u]}I_{|X_i|\leq b_n}]
\leq b_n^2\leq C\log(n/q_n)
\label{eq:t1-exp-7}
\end{equation}
It follows from equations~(\ref{eq:t1-exp-5})-(\ref{eq:t1-exp-7}) that
\begin{equation}
\mbox{lim sup}_{n\rightarrow\infty}
\frac{\sum_{\{i:\theta_{0i}\neq 0\}}\mathbb{E}_i[\mathbb{E}(\kappa_i|X_i)X_i]^2}
{q_n\log(n/q_n)}\leq C.
\label{eq:t1-exp-8}
\end{equation}
Finally, retracing the steps of equations~(\ref{eq:t1-exp-5})-(\ref{eq:t1-exp-7}), one can show that
\begin{equation}
\mbox{lim}_{n\rightarrow\infty}
\frac{\sum_{\{i:\theta_{0i}\neq 0\}}\mathbb{E}_i[\mathbb{E}(\kappa_i|X_i)X_i]^2]}
{q_n\log(n/q_n)}=0.
\label{eq:t1-exp-9}
\end{equation}
The theorem follows now from \eqref{eq:t1-exp-4}, \eqref{eq:t1-exp-8} and \eqref{eq:t1-exp-9}.\\

\noindent
\textbf{Proof of Theorem~\ref{thm:exp2}}. First observe that $\mathbb{E}[(\hat{\theta}_i-\theta_i)^2|X_i]=
\mathbb{E}[\sigma_{i}^2(1-\kappa_i)|X_i)$. Using the expression for the first
reciprocal moment of the inverse Gaussian distribution, 
$\mathbb{E}[\sigma_i^2|\kappa_i,X_i]=|X_i|(1+\lambda_i^2\tau^2)^{-1/2}+1)\leq|X_i|+1$, 
one gets
\[
\mathbb{E}_i[\sigma_{i}^2(1-\kappa_i)|X_i)]\leq \mathbb{E}_i(|X_i|+1)\leq
[(\theta_{0i}^2+\sigma_{0i}^2)^{1/2}+1]
\]
for $\theta_{0i}\neq 0$. Finally, for $\theta_{0i}=0$, following the
proof of Theorem~\ref{thm:exp1}, 
\begin{equation}
\mathbb{E}_i[\mathbb{E}(\sigma_{i}^2(1-\kappa_i)|X_i)]\leq\tau_n(\sigma_{0i}^2+\sigma_{0i}^4)
[\exp(\sigma_{0i}^2/2)+\exp(-\sigma_{0i}^2/2)].
\label{eq:t2-exp-1}
\end{equation}
Using the above results, the definition of $\tau_n$ and Markov's inequality,
$\mathbb{P}_{0i}[(\hat{\theta}_i-\theta_i)^2>q_n\log(n/q_n)]\leq 
C[\{q_n+(n-q_n)(q_n/n)^{1+\epsilon_n}\}/(q_n\log(n/q_n))\rightarrow 0$
as $n\rightarrow\infty$. This completes the proof of the theorem.\\

\section{Proofs: Inverse Gamma prior}\label{sec:proofs-inversegamma}

\noindent
Similar to the exponential prior case on error variance, the proofs of both theorems for the Inverse Gamma prior hinge crucially on three lemmas as well.\\

\begin{lemma}\label{lem:ig1}
Under the prior given in $\text{(i)-(iii)}$, for any $0<\tau<1$,
\[
\mathbb{E}(1-\kappa_i|X_i)\leq \tau^{2\eta}\exp(\alpha X_i^2/2\beta),
\]
where $0<\eta<\mbox{min}(a,1/2)$.\\
\end{lemma}

\begin{proof}
We begin with $X_i|\lambda_i^2,\sigma_i^2\sim\mbox{N}(0,\sigma_i^2
(1+\lambda_i^2\tau^2))$. This leads to
\[
\Pi(\lambda_i^2,\sigma_i^2|X_i) \propto
(\sigma_i^2)^{-\alpha - 3/2}(1+\lambda_i^2\tau^2)
^{-1/2} (\lambda_i^2)^{-a-1}
L(\lambda_i^2) \exp \left[- \frac{X_i^2 + 2\beta (1+\lambda_i^2\tau^2)}{2\sigma_i^2 (1+\lambda_i^2\tau^2)} \right].
\]
\noindent
The above leads to
\begin{equation}
\Pi(\lambda_i^2|X_i) \propto
(1+\lambda_i^2\tau^2)^{-1/2} (\lambda_i^2)^{-a-1} L(\lambda_i^2) \left[ \frac{X_i^2}{2(1+\lambda_i^2\tau^2)} + \beta \right]^{-\alpha - 1/2}
\label{eq:l1-ig-1}
\end{equation}
Hence,
\begin{eqnarray}
\mathbb{E}[1-\kappa_i|X_i]
&=&
\mathbb{E}[\lambda_i^2\tau^2(1+\lambda_i^2\tau^2)^{-1}|X_i] \nonumber \\
&=&
\frac{\tau^2 \int_0^\infty \lambda_i^2 (1+\lambda_i^2\tau^2)^{-3/2} (\lambda_i^2)^{-a-1} L(\lambda_i^2) [1 + X_i^2/\{ 2\beta(1+\lambda_i^2\tau^2) \}]^{-\alpha-1/2} d\lambda_i^2}{\int_0^\infty (1+\lambda_i^2\tau^2)^{-1/2} (\lambda_i^2)^{-a-1} L(\lambda_i^2) [1 + X_i^2/\{ 2\beta(1+\lambda_i^2\tau^2) \}]^{-\alpha-1/2} d\lambda_i^2} \nonumber \\
&\leq&
(1+X_i^2/2\beta)^{-\alpha-1/2} \tau^2 \frac{\int_0^\infty \lambda_i^2 (1+\lambda_i^2\tau^2)^{-3/2} (\lambda_i^2)^{-a-1} L(\lambda_i^2) d\lambda_i^2}{\int_0^\infty (1+\lambda_i^2\tau^2)^{-1/2} (\lambda_i^2)^{-a-1} L(\lambda_i^2) d\lambda_i^2} \nonumber \\
&\leq&
\tau^{2\eta} \exp(\alpha X_i^2/2\beta) N/D
\label{eq:l1-ig-2}
\end{eqnarray}

where $N=\int_{0}^{\infty}\lambda_i^2(1+\lambda_i^2\tau^2)^{-3/2}
(\lambda_i^2)^{-a-1}L(\lambda_i^2)d\lambda_i^2$ and 
$D=\int_{0}^{\infty}(1+\lambda_i^2\tau^2)^{1/2}(\lambda_i^2)^{-a-1}
L(\lambda_i^2)d\lambda_i^2$.\\

\noindent
Now repeating the steps of the proof of Lemma~\ref{lem:exp1}, for $0<\tau^2<1$, 
\begin{equation}
D\geq \int_0^\infty (1+\lambda_i^2)^{-1/2}(\lambda_i^2)^{-a-1}L(\lambda_i^2)d\lambda_i^2 = \mathbb{E}[(1+\lambda_i^2)^{-1/2}]
\label{eq:l1-ig-3}
\end{equation}
and with $0<\eta<\mbox{min}(a,1/2)$,
\begin{eqnarray}
N = K + M \tau^{-2+2\eta}/(a-\eta)
\label{eq:l1-ig-4}
\end{eqnarray}
\noindent
The lemma now follows from \eqref{eq:l1-ig-1}-\eqref{eq:l1-ig-4}.\\
\end{proof}

\begin{lemma}\label{lem:ig2}
For $0 < \tau < 1$, and for any $0<u<1$, $0<\delta<1$,
$$
\mathbb{E}[\kappa_i \mathbbm{1}_{\{ \kappa_i > u \}} | X_i] \leq C\tau^{-2a} |X_i|^{2a - 2\alpha - 1} (1+\sqrt{2}|X_i|)
$$
\end{lemma}

\begin{proof}
For $0 < \tau < 1$, $0<u<1$ and $0<\delta<1$, $\kappa_i > u \iff \lambda_i^2 < \frac{1-u}{u\tau^2}$
\begin{eqnarray}
\mathbb{E}[\kappa_i \mathbbm{1}_{\{ \kappa_i > u \}} | X_i]
&=&
\mathbb{E}[(1+\lambda_i^2 \tau^2)^{-1/2} \mathbbm{1}_{\{ \lambda_i^2 < \frac{1-u}{u\tau^2} \}} | X_i] \nonumber \\
&=&
\frac{\int_0^{(1-u)/u\tau^2} (1+\lambda_i^2\tau^2)^{-3/2} \left(1+\frac{X_i^2}{2\beta(1+\lambda_i^2\tau^2)} \right)^{-\alpha-1/2} (\lambda_i^2)^{-a-1} L(\lambda_i^2) d\lambda_i^2}{\int_0^\infty (1+\lambda_i^2\tau^2)^{-1/2} \left(1+\frac{X_i^2}{2\beta(1+\lambda_i^2\tau^2)} \right)^{-\alpha-1/2} (\lambda_i^2)^{-a-1} L(\lambda_i^2) d\lambda_i^2} \nonumber \\
&=&
\frac{N}{D}
\label{eq:l2-ig-1}
\end{eqnarray}
where $N = \int_0^{(1-u)/u\tau^2} (1+\lambda_i^2\tau^2)^{-3/2} \left(1+\frac{X_i^2}{2\beta(1+\lambda_i^2\tau^2)} \right)^{-\alpha-1/2} (\lambda_i^2)^{-a-1} L(\lambda_i^2) d\lambda_i^2$ and $D = \int_0^\infty (1+\lambda_i^2\tau^2)^{-1/2} \left(1+\frac{X_i^2}{2\beta(1+\lambda_i^2\tau^2)} \right)^{-\alpha-1/2} (\lambda_i^2)^{-a-1} L(\lambda_i^2) d\lambda_i^2$. Note that,
\begin{eqnarray}
N &\leq&
M \left( 1+\frac{uX_i^2}{2\beta} \right)^{-\alpha-1/2} \int_0^{(1-u)/u\tau^2} (1+\lambda_i^2\tau^2)^{-3/2} (\lambda_i^2)^{-a-1} L(\lambda_i^2) d\lambda_i^2 \nonumber \\
&\leq&
M \left( 1+\frac{uX_i^2}{2\beta} \right)^{-\alpha-1/2} \int_0^\infty (\lambda_i^2)^{-a-1} L(\lambda_i^2) d\lambda_i^2 \nonumber \\
&=&
M K^{-1} \left( 1+\frac{uX_i^2}{2\beta} \right)^{-\alpha-1/2}
\label{eq:l2-ig-2}
\end{eqnarray}
Again,
\begin{eqnarray}
D &\geq&
\int_{X_i^2/\tau^2}^{2X_i^2/\tau^2} (1+\lambda_i^2\tau^2)^{-1/2} \left(1+\frac{X_i^2}{2\beta(1+\lambda_i^2\tau^2)} \right)^{-\alpha-1/2} (\lambda_i^2)^{-a-1} L(\lambda_i^2) d\lambda_i^2 \nonumber \\
&\geq&
C \left(1+\frac{X_i^2}{2\beta(1+X_i^2)} \right)^{-\alpha-1/2} \int_{X_i^2/\tau^2}^{2X_i^2/\tau^2} (1+\lambda_i^2\tau^2)^{-1/2} (\lambda_i^2)^{-a-1} d\lambda_i^2 \nonumber \\
&\geq&
C \left(1+\frac{X_i^2}{2\beta(1+X_i^2)} \right)^{-\alpha-1/2} (1+2X_i^2)^{-1/2} \int_{X_i^2/\tau^2}^{2X_i^2/\tau^2} (\lambda_i^2)^{-a-1} d\lambda_i^2 \nonumber \\
&=&
C \left(1+\frac{X_i^2}{2\beta(1+X_i^2)} \right)^{-\alpha-1/2} (1+2X_i^2)^{-1/2} \left(\frac{X_i^2}{\tau^2}\right)^{-a} (1-2^{-a}) a^{-1}
\label{eq:l2-ig-3}
\end{eqnarray}
Combining the upper and lower bounds,
\begin{eqnarray}
\mathbb{E}[\kappa_i \mathbbm{1}_{\{ \kappa_i > u \}} | X_i]
&\leq&
\frac{M K^{-1}}{C} \left( 1+\frac{uX_i^2}{2\beta} \right)^{-\alpha-1/2} \left(1+\frac{X_i^2}{C(1+X_i^2)} \right)^{\alpha+1/2} (1+2X_i^2)^{1/2} \left(\frac{X_i^2}{\tau^2}\right)^{a} \frac{a}{1-2^{-a}} \nonumber \\
&=&
C\tau^{-2a} \left( 1+\frac{uX_i^2}{2\beta} \right)^{-\alpha-1/2} (1+2X_i^2)^{1/2} |X_i|^{2a} \nonumber \\
&\leq&
C\tau^{-2a} |X_i|^{2a - 2\alpha - 1} (1+\sqrt{2}|X_i|)
\label{eq:l2-ig-4}
\end{eqnarray}
\noindent
The lemma now follows from \eqref{eq:l2-ig-4}
\end{proof}

\begin{lemma}\label{lem:ig3}
For $0 < \tau < 1$, $\mathbb{E}[\kappa_i \mathbbm{1}_{\{ \kappa_i \leq u \}} | X_i] \leq C/X_i^2$, where $a+1<\alpha<a+2$.
\end{lemma}

\begin{proof}
Since $L$ is non decreasing, for $0<\kappa_i<u$, $(1-\kappa_i)/\kappa_i\tau^2$ is a decreasing function of $\kappa_i$. Now,
\begin{eqnarray}
\mathbb{E}[\kappa_i \mathbbm{1}_{\{ \kappa_i \leq u \}} | X_i]
&=&
\frac{\int_0^u (1-\kappa_i)^{-a-1} \kappa_i^{a+1/2} \left( 1+\frac{\kappa_iX_i^2}{2\beta} \right)^{-\alpha-1/2} L(\frac{1-\kappa_i}{\kappa_i\tau}) d\kappa_i}{\int_0^1 \kappa_i^{a-1/2} \left( 1+\frac{\kappa_iX_i^2}{2\beta} \right)^{-\alpha-1/2} L(\frac{1-\kappa_i}{\kappa_i\tau}) d\kappa_i} \nonumber \\
&\leq&
\frac{M(1-u)^{-a-1}}{C} \frac{\int_0^u \kappa_i^{a+1/2} \left( 1+\frac{\kappa_iX_i^2}{2\beta} \right)^{-\alpha-1/2} d\kappa_i}{\int_0^u \kappa_i^{a-1/2} \left( 1+\frac{\kappa_iX_i^2}{2\beta} \right)^{-\alpha-1/2} d\kappa_i} \nonumber \\
&=&
\frac{M(1-u)^{-a-1}}{C} \frac{\int_0^{uX_i^2/2\beta} (\frac{2\beta t}{X_i})^{a+1/2} (1+t)^{-\alpha-1/2} dt}{\int_0^{uX_i^2/2\beta} (\frac{2\beta t}{X_i})^{a-1/2} (1+t)^{-\alpha-1/2} dt} \nonumber \\
&=&
\frac{M(1-u)^{-a-1}}{C X_i^2} \frac{\int_0^{uX_i^2/2\beta} t^{a+1/2} (1+t)^{-\alpha-1/2} dt}{\int_0^{uX_i^2/2\beta} t^{a-1/2} (1+t)^{-\alpha-1/2} dt} \nonumber \\
&=&
\frac{M(1-u)^{-a-1}}{C X_i^2} A
\label{eq:l3-ig-1}
\end{eqnarray}
\noindent
where $A = \frac{\int_0^{uX_i^2/2\beta} t^{a+1/2} (1+t)^{-\alpha-1/2} dt}{\int_0^{uX_i^2/2\beta} t^{a-1/2} (1+t)^{-\alpha-1/2} dt}$. Consider the pdf $f_d(t) = \frac{t^{a-1/2} (1+t)^{-\alpha-1/2} \mathbbm{1}_{\{ t \leq d \}}}{\int_0^d t^{a-1/2} (1+t)^{-\alpha-1/2} dt}$.
\noindent
For $d_1<d_2$, $\frac{f_{d_2}(t)}{f_{d_1}(t)} = \frac{\int_0^{d_1} t^{a-1/2} (1+t)^{-\alpha-1/2} dt}{\int_0^{d_2} t^{a-1/2} (1+t)^{-\alpha-1/2} dt} \frac{\mathbbm{1}_{\{ t \leq d_2 \}}}{\mathbbm{1}_{\{ t \leq d_1 \}}}$. Note that $\frac{f_{d_2}(t)}{f_{d_1}(t)}$ is an increasing function of $t$. By the property of Monotone Likelihood Ratio,
\begin{equation}
A \leq \frac{\int_0^\infty t^{a+1/2} (1+t)^{-\alpha-1/2} dt}{\int_0^\infty t^{a-1/2} (1+t)^{-\alpha-1/2} dt} = \frac{\int_0^\infty (\frac{t}{1+t})^{a+1/2} (\frac{1}{1+t})^{a-\alpha+2} \frac{1}{(1+t)^2}dt}{\int_0^\infty (\frac{t}{1+t})^{a-1/2} (\frac{1}{1+t})^{a-\alpha+1} \frac{1}{(1+t)^2}dt} = \frac{B(a+3/2, a-\alpha+3)}{B(a+1/2, a-\alpha+2)}.
\label{eq:l3-ig-2}
\end{equation}
\noindent
The lemma now follows from \eqref{eq:l3-ig-1} and \eqref{eq:l3-ig-2}.
\end{proof}

\noindent
\textbf{Proof of Theorem~\ref{thm:ig1}}. Similar to the proof of Theorem 1 for the Laplace prior on $\sigma_i^2 (>0)$'s, we seperately consider two cases: $\{i:\theta_{0i}=0\}$ and $\{i:\theta_{0i}\neq 0\}$. First when $\theta_{0i}=0$, writing
$\mathbb{E}_{0i}$ to denote $\mathbb{E}_{0,\sigma_{0i}^2}$, one gets
\begin{equation}
\mathbb{E}_{0i}[\mathbb{E}(1-\kappa_i|X_i)X_i]^2 \leq \mathbb{E}_{0i}[X_i^2E(1-\kappa_i|X_i)]
\leq C\tau^{2\eta} \mathbb{E}_{0i}[X_i^2\exp(\frac{\alpha X_i^2}{2\beta})].
\label{eq:t1-ig-1}
\end{equation}
\noindent
Note that $\mathbb{E}_{0i}[X_i^2\exp(\alpha X_i^2/2\beta)]$ is finite if the signals remain within stable bounds, i.e. $l_1 \leq \inf_{1 \leq i \leq n} \sigma_{0i}^2 \leq \sup_{1 \leq i \leq n} \sigma_{0i}^2 \leq l_2$, where $l_2 \leq \beta/\alpha$ to ensure the Gaussian integral converges. Under this condition, we can bound the risk for null signals
\begin{equation}
\mathbb{E}[\hat\theta_{0i}^2] = \mathbb{E}_{0i}[E(1-\kappa_i|X_i)X_i]^2 \leq C\tau_n^{2\eta}
\label{eq:t1-ig-2}
\end{equation}
Thus, for the choice of $\tau_n = (q_n/n)^{\frac{1}{2\eta}(1+\epsilon_n)}$, it follows that
\begin{eqnarray}
\begin{aligned}
\frac{\sum\limits_{i: \theta_{0i}=0} \mathbb{E}_{0i}[E(1-\kappa_i|X_i)X_i]^2}{q_n \log (n/q_n)}
&\leq \frac{(n-q_n) C (q_n/n)^{(1+\epsilon_n)}}{q_n \log (n/q_n)} \\
&\leq \frac{(1-q_n/n) C (q_n/n)^{\epsilon_n}}{\log (n/q_n)}
\to 0
\label{eq:t1-ig-3}
\end{aligned}
\end{eqnarray}
\noindent
Next we consider the case when $\theta_{0i}\neq 0$. First writing $\mathbb{E}_i$ for $\mathbb{E}_{\theta_{0i},\sigma_{0i}^2}$, we use the inequality
\begin{eqnarray}
\mathbb{E}_i(\hat{\theta}_i-\theta_{0i})^2
&=&
\mathbb{E}_i[X_i-\theta_{0i} - \mathbb{E}(\kappa_i|X_i)X_i]^2 \nonumber \\
&\leq&
\sigma_{0i}^2 + \mathbb{E}_i[\mathbb{E}(\kappa_i|X_i)X_i^2] + 2\sigma_{0i}[\mathbb{E}_i[\mathbb{E}(\kappa_i|X_i)X_i^2]^{1/2}].
\label{eq:t1-ig-4}
\end{eqnarray}
\noindent
When $\kappa_i \leq u$, by Lemma~\ref{lem:ig3}, $\mathbb{E}[\kappa_i \mathbbm{1}_{\{ \kappa_i \leq u \}}|X_i] \leq C/X_i^2$, for some constant $C$ depending on $(\alpha, a)$. So, we have
\begin{equation}
\mathbb{E}[(\kappa_i \mathbbm{1}_{\{ \kappa_i \leq u \}}|X_i)X_i^2] \leq \frac{C}{X_i^2} \cdot X_i^2 = C
\label{eq:t1-ig-5}
\end{equation}
\noindent
Summing over the relevant cases,
\begin{equation}
\frac{\sum_{i: \theta_{0i} \neq 0} \mathbb{E}_i[\mathbb{E}[(\kappa_i \mathbbm{1}_{\{ \kappa_i \leq u \}}|X_i)X_i^2]]}{q_n \log (n/q_n)} \leq \frac{C q_n}{q_n \log (n/q_n)} \to 0 \text{ as } n \to \infty
\label{eq:t1-ig-6}
\end{equation}
When $\kappa_i > u$, by Lemma~\ref{lem:ig2}, $\mathbb{E}[\kappa_i \mathbbm{1}_{\{ \kappa_i > u \}}|X_i] \leq C\tau^{-2a} |X_i|^{2a-2\alpha-1} (1+\sqrt{2}|X_i|)$. We further split the $\kappa_i > u$ case into two sub-cases: $|X_i| \leq b_n$ and $|X_i|>b_n$, for some suitable $b_n$.\\

\noindent
When $|X_i| \leq b_n$, since $\kappa_i < 1$,
\begin{equation}
\mathbb{E}_i[\mathbb{E}(\kappa_i|X_i) X_i^2 \mathbbm{1}_{\{ \kappa_i > u \}} \mathbbm{1}_{\{ |X_i| \leq b_n \}}]
\leq \mathbb{E}_i[X_i^2 \mathbbm{1}_{\{ |X_i| \leq b_n \}}]
\leq \mathbb{E}_i[b_n^2 \mathbbm{1}_{\{ |X_i| \leq b_n \}}]
\leq b_n^2
\label{eq:t1-ig-7}
\end{equation}
Choosing $b_n^2 = M \log(n/q_n)$, where $M$ is chosen to be large enough such that $M/2\sigma_{0i}^2 > a/\eta$ for all $i$,
\begin{equation}
\frac{\sum_{i: \theta_{0i} \neq 0} \mathbb{E}_i[\mathbb{E}(\kappa_i|X_i) X_i^2 \mathbbm{1}_{\{ \kappa_i > u \}} \mathbbm{1}_{\{ |X_i| \leq b_n \}}]}{q_n \log (n/q_n)} \leq \frac{q_n b_n^2}{q_n \log (n/q_n)} = \frac{M \log (n/q_n)}{\log (n/q_n)} = M
\label{eq:t1-ig-8}
\end{equation}
\noindent
When $|X_i|>b_n$, using the result obtained in Lemma~\ref{lem:ig2}, under the condition $\alpha >a+1$, we can bound
\begin{eqnarray}
\mathbb{E}_i[\mathbb{E}(\kappa_i|X_i) X_i^2 \mathbbm{1}_{\{ \kappa_i > u \}} \mathbbm{1}_{\{ |X_i| > b_n \}}]
&\leq&
\mathbb{E}_i[X_i^2 C\tau_n^{-2a} |X_i|^{2a-2\alpha} \mathbbm{1}_{\{ |X_i| > b_n \}}] \nonumber \\
&=&
C\tau_n^{-2a} \mathbb{E}_i[|X_i|^{2a-2\alpha+2} \mathbbm{1}_{\{ |X_i| > b_n \}}] \nonumber \\
&=&
C\tau_n^{-2a} \int_{|x|>b_n} |x|^{2a-2\alpha+2} \frac{1}{(2\pi \sigma_{0i}^2)^{1/2}} \exp \left( -\frac{(x-\theta_{0i})^2}{2\sigma_{0i}^2} \right) dx \nonumber \\
\label{eq:t1-ig-9}
\end{eqnarray}
\noindent
The expectation term can be bounded by
\begin{equation}
\mathbb{E}_i[|X_i|^{2a-2\alpha+2} \mathbbm{1}_{\{ |X_i| > b_n \}}]
\leq C b_n^{2(a-\alpha+1)} \exp (-b_n^2/2 \sigma_{0i}^2)
\label{eq:t1-ig-10}
\end{equation}
\noindent
Finally, substituting the expressions for $b_n$ and $\tau_n$,
\begin{equation}
C \left( \frac{n}{q_n} \right)^{\frac{a}{\eta}(1+\epsilon_n)} \left[M \log (\frac{n}{q_n}) \right]^{a+\alpha -1} \left( \frac{q_n}{n} \right)^{\frac{M}{2 \sigma_{0i}^2}}
\to 0 \text{ as } n \to \infty
\label{eq:t1-ig-11}
\end{equation}
\noindent
The theorem then follows from \eqref{eq:t1-ig-3}, \eqref{eq:t1-ig-6} and \eqref{eq:t1-ig-11}.\\

\noindent
\textbf{Proof of Theorem~\ref{thm:ig2}}. First observe that $\mathbb{E}[(\hat{\theta}_i-\theta_i)^2|X_i]=
\mathbb{E}[\sigma_{i}^2(1-\kappa_i)|X_i]$. By Markov's inequality,
\begin{equation}
\sum_{i=1}^n \mathbb{P}_{0i}[(\hat\theta_i-\theta_i)^2 > M_n q_n \log(n/q_n)]
\leq \frac{\sum_{i=1}^n \mathbb{E}_{0i}[(\hat\theta_i-\theta_i)^2]}{M_n q_n \log(n/q_n)}
= \frac{\sum_{i=1}^n \mathbb{E}_{0i}[\mathbb{E}[\sigma_{i}^2(1-\kappa_i)|X_i]]}{M_n q_n \log(n/q_n)}
\label{eq:t2-ig-1}
\end{equation}
\noindent
We separately consider two cases: $\{i:\theta_{0i}=0\}$ and $\{i:\theta_{0i}\neq 0\}$. First when $\{i:\theta_{0i}\neq 0\}$, since $1-\kappa_i<1$, we have for some constant $C$,
\begin{equation}
\mathbb{E}_{0i}[\mathbb{E}[\sigma_{i}^2(1-\kappa_i)|X_i]]
\leq \frac{\beta}{\alpha-1/2} + \frac{\mathbb{E}_{0i}[X_i^2]}{2\alpha-1}
= \frac{\beta}{\alpha-1/2} + \frac{\theta_{0i}^2+\sigma_{0i}^2}{2\alpha-1}
\leq C
\label{eq:t2-ig-2}
\end{equation}
\noindent
When $\{\theta_{0i}=0\}$, using Lemma~\ref{lem:ig1},
\begin{eqnarray}
\mathbb{E}_{0i}[\mathbb{E}[\sigma_i^2(1-\kappa_i)|X_i]]
&\leq&
\frac{\beta}{\alpha-1/2} \mathbb{E}_{0i}[\mathbb{E}[1-\kappa_i|X_i]] + \frac{1}{2\alpha-1} \mathbb{E}_{0i}[X_i^2\mathbb{E}[1-\kappa_i|X_i]] \nonumber \\
&\leq&
\left( \frac{\beta}{\alpha-1/2} + \frac{1}{2\alpha-1} \right) C\tau_n^{2\eta} \mathbb{E}_{0i}[X_i^2 \exp(\alpha X_i^2/C_1)]
\label{eq:t2-ig-3}
\end{eqnarray}
\noindent
Note that $\mathbb{E}_{0i}[X_i^2 \exp(\alpha X_i^2/C)]$ can be bounded by a constant, say $K_1$, by our assumption $\mbox{max}_{1\leq i\leq n}\sigma_{0i}^2=O(1)$. Now,
\begin{equation}
\sum_{i=1}^n \mathbb{E}_{0i}[\mathbb{E}[\sigma_{i}^2(1-\kappa_i)|X_i]]
\leq
(n-q_n) C [(q_n/n)^{\frac{1}{2\eta}(1+\epsilon_n)}]^{2\eta}
=
C q_n (1 - q_n/n) (q_n/n)^{2\epsilon}
\label{eq:t2-ig-4}
\end{equation}
\noindent
Finally, summing over the relevant cases,
\begin{eqnarray}
\frac{\sum_{i:\theta_{0i}=0} \mathbb{E}_{0i}[\mathbb{E}[\sigma_{i}^2(1-\kappa_i)|X_i]]}{M_n q_n \log(n/q_n)}
&\leq&
\frac{C q_n + C q_n (1-q_n/n) (q_n/n)^{\epsilon_n}}{M_n q_n \log (n/q_n)} \nonumber \\
&=&  \frac{C + C (1-q_n/n) (q_n/n)^{\epsilon_n}}{M_n \log (n/q_n)} \nonumber \\
&\to& 0 \text{ as } n \to \infty
\label{eq:t2-ig-5}
\end{eqnarray}
As $n \to \infty$, since $q_n = o(n)$ and $M_n \to \infty$, the numerator remains bounded while the denominator $M_n \log(n/q_n) \to \infty$. Thus, the expression tends to $0$.\\
\noindent
The theorem then follows from \eqref{eq:t2-ig-1}, \eqref{eq:t2-ig-2} and \eqref{eq:t2-ig-5}.

\section{Final Remarks}\label{sec:remarks}

\noindent
The paper seems to be the first attempt to develop theoretical results for
global-local priors with unknown sample variance. 
There are many possible avenues for extending these results. 
A natural extension is to consider the regression framework and see possible extensions of the present 
results.\\

\backmatter

\bibliography{sn-bibliography}

\end{document}